\documentclass[12pt, a4paper]{article}
\usepackage{amsmath}   
\usepackage{comment}
\usepackage{amsfonts}
\usepackage{amssymb}   
\usepackage{epsfig}
\usepackage{graphicx}
\usepackage{amsthm}
\usepackage{newlfont}
\usepackage{dsfont}
\usepackage{eufrak}
\usepackage{color}


\newtheorem{thm}{Theorem}[section]
\newtheorem{lem}[thm]{Lemma}
\newtheorem{cor}[thm]{Corollary}
\newtheorem{pro}[thm]{Proposition}
\newtheorem{con}[thm]{Conjecture}
\newtheorem{obs}[thm]{Observation}
\theoremstyle{definition}
\newtheorem{den}[thm]{Definition}
\newtheorem{exa}[thm]{Example}

\theoremstyle{remark}
\newtheorem*{rem}{Remark}

\renewcommand{\labelenumi}{(\roman{enumi})}
\graphicspath{{figures/}}
\begin{document}
\title{Ultra log-concavity and real-rootedness of dependence polynomials}

\footnotetext[1]{The work is partially supported by the National Natural Science Foundation of China (Grant No. 12671410).}

\author{Yan-Ting Xie, ~Shou-Jun Xu$\thanks{Corresponding author. E-mail address: ~shjxu@lzu.edu.cn (S.-J. Xu).}$}

\date{\small  School of Mathematics and Statistics, Gansu Center for Applied Mathematics, \\Lanzhou University, Lanzhou, Gansu 730000, China}

\maketitle

\begin{abstract}

For some positive integer  $m$, a real polynomial $P(x)=\sum\limits_{k=0}^ma_kx^k$ with $a_k\geqslant 0$  is called {log-concave} (resp. ultra log-concave) if $a_k^2\geqslant a_{k-1}a_{k+1}$ (resp. $a_k^2\geqslant \left(1+\frac{1}{k}\right)\left(1+\frac{1}{m-k}\right)\cdot$ $a_{k-1}a_{k+1}$) for all $1\leqslant k\leqslant m-1$. If $P(x)$ has only real roots, then it is called {real-rooted}. It is well-known that the conditions of log-concavity, ultra log-concavity and real-rootedness are ever-stronger. 

For a graph $G$, a dependent set is a set of vertices which is not independent, i.e., the set of vertices whose induced subgraph contains at least one edge. The dependence polynomial of $G$ is defined as $D(G, x):=\sum\limits_{k\geqslant 0}d_k(G)x^k$, where $d_k(G)$ is the number of dependent sets of size $k$ in $G$.  Horrocks proved that $D(G, x)$ is log-concave for every graph $G$ [J. Combin. Theory, Ser. B, 84 (2002) 180--185]. In the present paper, we prove that, for a graph $G$, $D(G, x)$ is ultra log-concave if $G$ is $(K_2\cup 2K_1)$-free or contains an independent set of size $|V(G)|-2$, and give the characterization of graphs whose dependence polynomials are real-rooted. Finally, we focus more attention to the problems of log-concavity about independence systems and pose several conjectures closely related the famous Mason's Conjecture.

\setlength{\baselineskip}{17pt}
{} \vskip 0.1in \noindent%
\end{abstract}
\section{Introduction}
Throughout this paper, except where stated otherwise, all graphs we consider are undirected, finite and simple. 

Let $G$ be a graph with vertex set  $V(G)$ and edge set $E(G)$. 
An {\em independent set} of $G$ is a subset of $V(G)$ whose induced subgraph in $G$ contains no edges. Let $i_k(G)$ be the number of independent sets of size $k$ in $G$. In particular, $i_0(G)=1$. The generating polynomial of $\{i_k(G)\}$,
\begin{equation*}
I(G,x):=\sum\limits_{k\geqslant 0}i_{k}(G)x^k,
\end{equation*}
is called the {\em independence polynomial} of $G$ \cite{gh83}.

If a subset of $V(G)$ is not independent, we call it a {\em dependent set}, i.e., a dependent set is a subset of $V(G)$ whose induced subgraph in $G$ contains at least one edge. Let $d_k(G)$ be the number of dependent sets of size $k$ in $G$. In particular, $d_0(G)=d_1(G)=0$. Similarly, the generating polynomial of $\{d_k(G)\}$,
\begin{equation*}
D(G,x):=\sum\limits_{k\geqslant 0}d_{k}(G)x^k,
\end{equation*}
is called the {\em dependence polynomial} of $G$. 

More generally, a graph polynomial is a graph invariant with values in a polynomial ring, usually a subring of real polynomial ring $\mathds{R}[x]$. Like independence polynomials and dependence polynomials, many of graph polynomials are the generating polynomials for the sequences of numbers of subsets of vertex or edge set with some given properties. A {\em graph property} $\mathcal{A}$ is defined as a family of graphs closed under isomorphism. Let $G$ be a graph. For $S\subseteq V(G)$, $G[S]$ denotes the subgraph induced by $S$. Let $\mathcal{A}$ be a graph property and $a_k(G)$ the number of induced subgraphs of order $k$ of $G$ which are in $\mathcal{A}$. The graph polynomial associated with the graph $G$ and the graph property $\mathcal{A}$, denoted by $P_{\mathcal{A}}(G,x)$, is defined as: $P_{\mathcal{A}}(G,x):=\sum\limits_{S\subseteq V(G),G[S]\in\mathcal{A}}x^{|S|}=\sum\limits_{k=0}^{|V(G)|}a_k(G)x^k$. Correspondingly, when $\mathcal{A}$ consists of all graphs with no edges, $P_{\mathcal{A}}(G,x)$ is the independence polynomial of $G$; when $\mathcal{A}$ consists of all graphs containing at least one edge, $P_{\mathcal{A}}(G,x)$ is the dependence polynomial of $G$. The two more examples are: when $\mathcal{A}$ consists of all complete graphs, $P_{\mathcal{A}}(G,x)$ is the clique polynomial of $G$ \cite{hl94}; when $\mathcal{A}$ consists of all forests, $P_{\mathcal{A}}(G,x)$ is the acyclic polynomial of $G$ \cite{bbp22}.

Let $G$ be a graph and $\mathcal{A}$ a graph property. The {\em complement} of $\mathcal{A}$, denoted by $\bar{\mathcal{A}}$, is the graph property satisfying that $G\in\bar{\mathcal{A}}\Longleftrightarrow G\not\in\mathcal{A}$. About the graph polynomials associated with the graph properties $\mathcal{A}$ and $\bar{\mathcal{A}}$, we have
\begin{pro}\label{pro:Complementary}
For any graph $G$ and graph property $\mathcal{A}$, $P_{\mathcal{A}}(G,x)+P_{\bar{\mathcal{A}}}(G,x)=(x+1)^{|V(G)|}$.
\end{pro}
\begin{proof}
For any integer $0\leqslant k\leqslant |V(G)|$, denote $\mathcal{A}_k(G):=\{S\subseteq V(G):G[S]\in\mathcal{A},|S|=k\}$ (resp. $\bar{\mathcal{A}}_k(G):=\{S\subseteq V(G):G[S]\in\bar{\mathcal{A}},|S|=k\}$) and $a_k(G):=|\mathcal{A}_k(G)|$ (resp. $\bar{a}_k(G):=|\bar{\mathcal{A}}_k(G)|$). Since $\mathcal{A}_k(G)$ and $\bar{\mathcal{A}}_k(G)$ are disjoint, $a_k(G)+\bar{a}_k(G)=|\mathcal{A}_k(G)\cup \bar{\mathcal{A}}_k(G)|=|\{S\subseteq V(G):|S|=k\}|={|V(G)| \choose k}$. Combined with $P_{\mathcal{A}}(G,x)=\sum\limits_{k=0}^{|V(G)|}a_k(G)x^k$ and $P_{\bar{\mathcal{A}}}(G,x)=\sum\limits_{k=0}^{|V(G)|}\bar{a}_k(G)x^k$, $P_{\mathcal{A}}(G,x)+P_{\bar{\mathcal{A}}}(G,x)=\sum\limits_{k=0}^{|V(G)|}a_k(G)x^k+\sum\limits_{k=0}^{|V(G)|}\bar{a}_k(G)x^k=\sum\limits_{k=0}^{|V(G)|}{|V(G)| \choose k}x^k=(x+1)^{|V(G)|}$.
\end{proof}

We call the graph property $\mathcal{A}$ {\em hereditary} if it is closed under taking induced subgraphs, i.e., for any graphs $G$, $H$, if $G\in\mathcal{A}$ and $H$ is an induced subgraph of $G$, then $H\in\mathcal{A}$. In contrast, replacing the condition `$H$ is an induced subgraph of $G$' by `$G$ is an induced subgraph of $H$', we say that $\mathcal{A}$ is {\em co-hereditary}. It is obvious that $\mathcal{A}$ is co-hereditary if and only if $\bar{\mathcal{A}}$ is hereditary.

For example, if $\mathcal{A}$ is the graph property of containing at least one edge, then $\bar{\mathcal{A}}$ is the graph property of containing no edges. We can see that $\mathcal{A}$ is co-hereditary and $\bar{\mathcal{A}}$ is hereditary. Moreover, by Proposition \ref{pro:Complementary} and the definitions of dependence and independence polynomials, for any graph $G$, $D(G,x)+I(G,x)=(x+1)^{|V(G)|}$.

Now, we introduce the concepts of unimodality, log-concavity, real-rootedness, and so on. Let $\{a_k\}_{k=0}^m$ be a sequence with real numbers. We define:
\begin{itemize}
\item {\bf Unimodality:} There exists an integer $k$ ($0\leqslant k\leqslant m$) satisfying that $a_0\leqslant a_1\leqslant\cdots\leqslant a_k\geqslant a_{k+1}\geqslant\cdots\geqslant a_m$.
\item {\bf Logarithmically concavity (or log-concavity for short):} $a_k^2\geqslant a_{k-1}a_{k+1}$ for $1\leqslant k\leqslant m-1$.
\item {\bf Ordered log-concavity:} $a_k^2\geqslant \left(1+\frac{1}{k}\right) a_{k-1}a_{k+1}$ for $1\leqslant k\leqslant m-1$ (equivalent to the log-concavity of the sequence $\{k!a_k\}_{k=0}^m$).
\item {\bf Ultra log-concavity:} $a_k^2\geqslant \left(1+\frac{1}{k}\right)\left(1+\frac{1}{m-k}\right) a_{k-1}a_{k+1}$ for $1\leqslant k\leqslant m-1$ (equivalent to the log-concavity of the sequence $\left\{a_k\left/{m\choose k}\right.\right\}_{k=0}^m$).
\item {\bf Real-rootedness:} The polynomial $\sum\limits_{k=0}^m a_kx^k$ has all real roots.
\end{itemize} 
For the non-negative sequences, it is easy to see that these conditions become ever-stronger starting from the second one (see \cite{bg21}). Moreover, adding the condition of having no internal zeros (that is, for any $0\leqslant i<j<k\leqslant m$, $a_i,a_k>0\Longrightarrow a_j>0$), the log-concavity implies the unimodality. A polynomial is called {\em unimodal} (resp. {\em log-concave}, {\em ordered log-concave}, {\em ultra log-concave}) if the sequence of its coefficients is unimodal (resp. log-concave, ordered log-concave, ultra log-concave). 

Let $G$ be a graph and $\mathcal{A}$ a graph property. Recently, Makowsky and Rakita \cite{mr23} obtained the following results about $P_{\mathcal{A}}(G,x)$:
\begin{enumerate}
\renewcommand{\labelenumi}{\arabic{enumi}.}
\item If $\mathcal{A}$ is co-hereditary, $P_{\mathcal{A}}(G,x)$ is unimodal for almost every graph $G\in\mathcal{G}(n,p)$ for every edge probability $p\in (0,1)$.
\item If $\mathcal{A}$ is hereditary with a graph $G\in\mathcal{A}$ which is neither a clique nor an edgeless graph, then $P_{\mathcal{A}}(G,x)$ is real-rooted if and only if $G\in\mathcal{A}$.
\end{enumerate}

It is a natural problem what we can say about the log-concavity or the real-rootedness of $P_{\mathcal{A}}(G,x)$ when $\mathcal{A}$ is co-hereditary. In this paper, we consider this problem on dependence polynomials. As mentioned above, dependence polynomials and independence polynomials are complementary, but for the problem of unimodality, log-concavity or real-rootedness, the independence polynomials are more difficult than the dependence polynomials. The independence polynomials are not necessarily log-concave or unimodal \cite{amse87}. Even for bipartite graphs \cite{bk13} and well-covered graphs \cite{mt03}, there are examples that the independence polynomials are not unimodal. More recently, the examples of trees with non-log-concave independence polynomials were constructed \cite{kl23,klym23}. For the real-rootedness, only very few independence polynomials of graph classes can be proved having only real roots, like claw-free graphs \cite{cs07}. In contrast, for dependence polynomials, Horrocks \cite{h02} showed that $D(G,x)$ is log-concave for every graph $G$.  However, the stronger properties of dependence polynomials, like ordered log-concavity, ultra log-concavity and real-rootedness,  have never been studied.

If $G$ contains no edges, there are no dependent sets in $G$, i.e., $D(G,x)\equiv 0$. Thus, we only consider the graphs which contains at least one edge in what follows. 

We denote the {\em complete graph}, the {\em complete bipartite graph}, the {\em path} and the {\em cycle} by $K_n$, $K_{m,n}$, $P_n$ and $C_n$, respectively. In particular, $K_{1,n}$ is called a {\em star}. Let $G$ and $H$ be two disjoint graphs. The union $G\cup H$ is the graph with the vertex set $V(G\cup H)=V(G)\cup V(H)$ and the edge set $E(G\cup H)=E(G)\cup E(H)$. For a graph $G$ and $n\in\mathds{N}^*$, $nG$ denotes the disjoint union of $n$ copies of $G$. For a graph $H$, a graph $G$ is called {\em $H$-free} if no induced subgraphs of $G$ are isomorphic to $H$. About the ultra log-concavity of dependence polynomials, we have  
\begin{thm}\label{thm:ULC}
Let $G$ be a graph of order $n$ containing at least one edge. If $G$ is a $K_2\cup 2K_1$-free graph or has an independent set of size $n-2$, then the dependence polynomial $D(G,x)$ is ultra log-concave.
\end{thm}
\begin{rem}
The complement of $K_2\cup 2K_1$ is $K_4-e$, which is called a {\em diamond}, so the $K_2\cup 2K_1$-free graphs are equivalent to the complements of the diamond-free graphs. 
\end{rem}

Let $G_1,G_2$ be the two graphs illustrated in Fig. \ref{fig:G1G2}. We give the characterization of graphs with real-rooted dependence polynomials as follows:
\begin{thm}\label{thm:RR}
Let $G$ be a graph of order $n$ which contains at least one edge. The dependence polynomial $D(G,x)$ is real-rooted if and only if $G\in\{K_2\cup (n-2)K_1,P_3\cup (n-3)K_1,K_3\cup (n-3)K_1,2K_2\cup (n-4)K_1,P_4\cup (n-4)K_1,C_4\cup (n-4)K_1,G_1\cup (n-4)K_1,G_2\cup (n-5)K_1\}$.
\end{thm}

\begin{figure}[!htbp]
\centering
\scalebox{0.5}[0.5]{\includegraphics{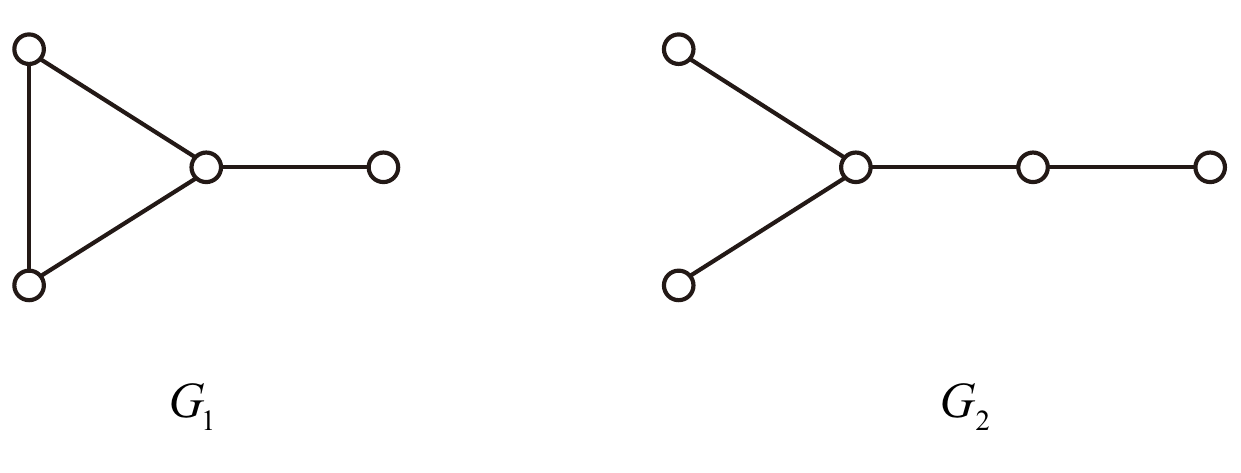}}
\caption{$G_1$ and $G_2$.}{\label{fig:G1G2}}
\end{figure}

For the graphs without isolated vertices, we have the following corollary:
\begin{cor}\label{cor:GnoIsolatedVertices}
Let $G$ be a graph without isolated vertices. The dependence polynomial $D(G,x)$ is real-rooted if and only if $G\in\{K_2,P_3,K_3,2K_2,P_4,C_4,G_1,G_2\}$.
\end{cor}

The paper is organized as follows. We will prove Theorem \ref{thm:ULC} in Section 2 and Theorem \ref{thm:RR} in Section 3. Finally, in Section 4, we conclude the paper and give some discussion on the log-concave problems of independence systems, which generalize Mason's Conjecture.
\section{Proof of Theorem \ref{thm:ULC}}
Let $G$ be a graph of order $n$. For $S\subseteq V(G)$, let's denote $G-S:=G[V(G)\setminus S]$. $S$ is called a {\em co-dependent set} of $G$ if $V(G)\setminus S$ is a dependent set of $G$. The number of co-dependent sets of size $k$ in $G$ is denoted by $\bar{d}_k(G)$. It is obvious that $\bar{d}_k(G)=d_{n-k}(G)$ and $\bar{d}_n(G)=\bar{d}_{n-1}(G)=0$. By the definition of ultra log-concavity,  $D(G,x)$ is ultra log-concave if and only if the sequence $\{\bar{d}_k(G)\}_{k=0}^n$ is ultra log-concave. Thus, we consider the co-dependent sets of $G$ in what follows. We will write $\bar{d}_k(G)$ as $\bar{d}_k$ shortly without causing confusion.

Let $\overline{\mathcal{D}}(G)$ ($\overline{\mathcal{D}}$ without causing confusion) be the family of co-dependent sets in $G$. We can see that $(V(G),\overline{\mathcal{D}})$ is an {\em independence system} (or a {\em simplicial complex} in a view of geometry), i.e., for any $S,T\subseteq V(G)$, if $S\subseteq T$ and $T\in\overline{\mathcal{D}}$, then $S\in\overline{\mathcal{D}}$. And further, co-dependent sets satisfy the following property:
\begin{obs}\label{obs:ContractionProperty}
Let $S$ be a co-dependent set of a graph $G$ and $T$ a subset of $V(G)\setminus S$. $T$ is a co-dependent set of $G-S$ if and only if $S\cup T$ is a co-dependent set of $G$.
\end{obs}
The main idea of the proof of Theorem \ref{thm:ULC} is analogous to that presented in \cite{algv18}. Define the bivariate function:
\begin{equation*}
f_G(y,z):=\sum_{k=0}^{n-2}\bar{d}_ky^{n-k}z^k.
\end{equation*}

Let $\partial_y$ (resp. $\partial_z$, $\partial_{z_i}$, etc.) denote the partial derivative operator that maps a multivariate polynomial to its partial derivative with respect to $y$ (resp. $z$, $z_i$, etc.). Denote the quadratic form
\begin{eqnarray}
q_k(y,z)&:=&\frac{\partial_y^{n-k-1}\partial_z^{k-1}f_G}{(k-1)!}\nonumber\\
&=&\frac{(n-k+1)!}{2}\bar{d}_{k-1}y^2+(n-k)!k\bar{d}_kyz+\frac{(n-k-1)!k(k+1)}{2}\bar{d}_{k+1}z^2\label{eq:qk}
\end{eqnarray}
for $1\leqslant k\leqslant n-3$.

For a function $f\in\mathds{R}[z_1,z_2,\cdots,z_n]$, the {\em Hessian matrix} (or shortly, {\em Hessian}), denoted by $\nabla^2f$, is an $n\times n$ symmetric matrix whose $(i,j)$-th entry is $(\nabla^2f)_{i,j}:=\partial_{z_i}\partial_{z_j}f$. In particular, if $f$ is a quadratic form, then its matrix representation is $f=\frac{1}{2}\mathbf{z}^{\mathsf{T}}\nabla^2f\mathbf{z}$, where $\mathbf{z}:=(z_1,z_2,\cdots,z_n)$.  Now, we prove the following lemma.
\begin{lem}\label{lem:Hessian<=0}
Let $G$ be a graph of order $n$ containing at least one edge. If $G$ is a $K_2\cup 2K_1$-free graph or has an independent set of size $n-2$, then the determinant of the Hessian $\det(\nabla^2q_k)\leqslant 0$ for $1\leqslant k\leqslant n-3$.
\end{lem}
\begin{proof}
 W.l.o.g., denote $V(G)=[n]=\{1,2,\cdots,n\}$. We define the homogeneous polynomial $g_G\in\mathds{R}[y,z_1,z_2,\cdots,z_n]$ as follows:
\begin{equation*}
g_G(y,z_1,z_2,\cdots,z_n):=\sum_{S\in\overline{\mathcal{D}}}y^{n-|S|}\prod_{i\in S}z_i.
\end{equation*}

For $S\subseteq V(G)$, denote the differential operator $\partial^S:=\prod\limits_{i\in S}\partial_{z_i}$. Let $S$ be a co-dependent set of size $k-1$ in $G$. Denote $V(G)\setminus S:=\{l_1,l_2,\cdots,l_{n-k+1}\}$. We consider $q_k^S:=\partial_y^{n-k-1}\partial^Sg_G$, which is a quadratic form in variables $y,z_{l_1},z_{l_2},\cdots,z_{l_{n-k+1}}$.  Let $\overline{\mathcal{D}}_i(G)$ denote the family of the co-dependent sets of size $i$ in $G$. By Observation \ref{obs:ContractionProperty}, $S'\supset S$ and $S'\in\overline{\mathcal{D}}_k(G)$ (resp. $S'\in\overline{\mathcal{D}}_{k+1}(G)$) if and only if $S'\setminus S\in\overline{\mathcal{D}}_1(G-S)$ (resp. $S'\setminus S\in\overline{\mathcal{D}}_2(G-S)$). Then
\begin{eqnarray}
q_k^S(y,z_{l_1},z_{l_2},\cdots,z_{l_{n-k+1}})&:=&\partial_y^{n-k-1}\partial^Sg_G\nonumber\\[6pt]
&=&\frac{(n-k+1)!}{2}y^2+(n-k)!\sum_{\{i\}\in\overline{\mathcal{D}}_1(G-S)}yz_i\nonumber\\
&&{}+(n-k-1)!\sum_{\{i,j\}\in\overline{\mathcal{D}}_2(G-S)}z_iz_j\nonumber\\
&=&\frac{1}{2}\mathbf{\hat{z}}_S^{\mathsf{T}}\left(\nabla^2q_k^S\right)\mathbf{\hat{z}}_S,\label{eq:qS}
\end{eqnarray}
where $\mathbf{\hat{z}}_S:=(y,z_{l_1},z_{l_2},\cdots,z_{l_{n-k+1}})^{\mathsf{T}}$. Let $\mathbf{a}=(a_1,a_2,\cdots,a_{n-k+1})^{\mathsf{T}}$ be a $(0,1)$-vector where $a_i=1$ if $\{l_i\}\in\overline{\mathcal{D}}_1(G-S)$ and $a_i=0$ otherwise, $B=\left(b_{ij}\right)$ an $(n-k+1)\times (n-k+1)$ symmetric $(0,1)$-matrix where $b_{ij}=1$ if $\{l_i,l_j\}\in\overline{\mathcal{D}}_2(G-S)$ and $b_{ij}=0$ otherwise. Then
\begin{equation*}
\nabla^2q_k^S=(n-k-1)!\left[\begin{array}{cc}
(n-k)(n-k+1) & (n-k)\mathbf{a}^{\mathsf{T}}\\
(n-k)\mathbf{a} & B
\end{array}\right].
\end{equation*}

Combined with (\ref{eq:qk}) and (\ref{eq:qS}), we can obtain: 
\begin{equation}\label{eq:qkqS}
q_k(y,z)=\sum\limits_{S\in\overline{\mathcal{D}}_{k-1}}q_k^S(y,z,z,\cdots,z).
\end{equation}

Substituting $z_{l_1}=z,z_{l_2}=z,\cdots,z_{l_{n-k+1}}=z$ into (\ref{eq:qS}), we can obtain:
\begin{eqnarray}
&&q_k^S(y,z,z,\cdots,z)=\nonumber\\
&&\frac{(n-k-1)!}{2}[y\ \ z]\left[\begin{array}{cc}
(n-k)(n-k+1) & (n-k)\sum\limits_{1\leqslant i\leqslant n-k+1} a_i\\
(n-k)\sum\limits_{1\leqslant i\leqslant n-k+1} a_i & \sum\limits_{1\leqslant i,j\leqslant n-k+1} b_{ij}
\end{array}\right]\left[\begin{array}{c}
y\\
z
\end{array}\right].\label{eq:QS}
\end{eqnarray}

Denote the $2\times 2$ matrix $Q_S:=\left[\begin{array}{cc}
(n-k)(n-k+1) & (n-k)\sum\limits_{1\leqslant i\leqslant n-k+1} a_i\\
(n-k)\sum\limits_{1\leqslant i\leqslant n-k+1} a_i & \sum\limits_{1\leqslant i,j\leqslant n-k+1} b_{ij}
\end{array}\right].$ Combined with (\ref{eq:qkqS}), (\ref{eq:QS}) and the fact that $q_k(y,z)=\frac{1}{2}\cdot[y\ \ z]\nabla^2q_k\left[\begin{array}{c}
y\\
z
\end{array}\right]$, we obtain:
\begin{equation}\label{eq:sum}
\nabla^2q_k=(n-k-1)!\sum\limits_{S\in\overline{\mathcal{D}}_{k-1}}Q_S.
\end{equation}

For $S\in\overline{\mathcal{D}}_{k-1}(G)$, let's consider the subgraph $G-S$. Since $S$ is a co-dependent set of $G$, $G-S$ must contain at least one edge. We call a vertex $i$ in $V(G)\setminus S$ a {\em co-dependent vertex} in $G-S$ if $\{i\}\in\overline{\mathcal{D}}_1(G-S)$, and a pair of vertices $\{i,j\}$ a {\em co-dependent pair} if $\{i,j\}\in\overline{\mathcal{D}}_2(G-S)$. By the definitions, $\sum\limits_{1\leqslant i\leqslant n-k+1} a_i$ is the number of co-dependent vertices in $G-S$ and $\sum\limits_{1\leqslant i,j\leqslant n-k+1} b_{ij}$ is twice of the number of co-dependent pairs in  $G-S$.  If a vertex $i\in V(G)\setminus S$ is not a co-dependent vertex in $G-S$, we call $i$ a {\em co-independent vertex} in $G-S$. The definition of {\em co-independent pairs} is similar. 

\noindent{\bf Claim.} The number of co-independent vertices in $G-S$ is at most 2.

By contradiction, w.l.o.g., assume that $l_1$, $l_2$ and $l_3$ are three co-independent vertices in $G-S$. By the definition, all of $\{l_2,l_3,l_4,\cdots,l_{n-k+1}\}$, $\{l_1,l_3,l_4,\cdots,l_{n-k+1}\}$ and $\{l_1,l_2,l_4,\cdots,l_{n-k+1}\}$ are independent sets in $G-S$. This deduces that $G-S$ contains no edges, a contradiction.   

Now, we distinguish cases to discuss $Q_S$.

{\bf Case 1.} There are two co-independent vertices in $G-S$.

W.l.o.g., assume that $l_1$ and $l_2$ are co-independent vertices. Then $\{l_2,l_3,l_4,\cdots,l_{n-k+1}\}$ and $\{l_1,l_3,l_4,$ $\cdots,l_{n-k+1}\}$ are independent sets and further $G-S\cong K_2\cup (n-k-1)K_1$, which only $l_1$ and $l_2$ are adjacent. In this case, $\sum\limits_{1\leqslant i\leqslant n-k+1} a_i=n-k-1$, $\sum\limits_{1\leqslant i,j\leqslant n-k+1} b_{ij}=2\cdot{n-k-1\choose 2}=(n-k-1)(n-k-2)$. So $Q_S=\left[\begin{array}{cc}
(n-k)(n-k+1) & (n-k)(n-k-1)\\
(n-k)(n-k-1) & (n-k-1)(n-k-2)
\end{array}\right].$ For convenience, we denote 
$$H_1=\left[\begin{array}{cc}
(n-k)(n-k+1) & (n-k)(n-k-1)\\
(n-k)(n-k-1) & (n-k-1)(n-k-2)
\end{array}\right]:=\left[\begin{array}{cc}
\alpha_1 & \beta_1\\
\beta_1 & \gamma_1
\end{array}\right].$$

{\bf Case 2.} There is exactly one co-independent vertex in $G-S$.

In this case, $G-S\cong K_{1,p}\cup (n-k-p)K_1$ ($p\geqslant 2$) and $\sum\limits_{1\leqslant i\leqslant n-k+1} a_i=n-k$. W.l.o.g., assume $l_1$ is the center of $K_{1,p}$.

{\bf Subcase 2.1.} $p=2$.

Assume $l_2,l_3$ are the leaves of the $K_{1,2}$. In this subcase, the co-independent pairs of $G-S$ are $\{l_2,l_3\}$ and $\{l_1,l_i\}$ ($2\leqslant i\leqslant n-k+1$). Thus, $\sum\limits_{1\leqslant i,j\leqslant n-k+1} b_{ij}=2\left({n-k+1 \choose 2}-(n-k)-1\right)=(n-k)(n-k-1)-2$. So $Q_S=\left[\begin{array}{cc}
(n-k)(n-k+1) & (n-k)^2\\
(n-k)^2 & (n-k)(n-k-1)-2
\end{array}\right].$ Let's denote 
$$H_2=\left[\begin{array}{cc}
(n-k)(n-k+1) & (n-k)^2\\
(n-k)^2 & (n-k)(n-k-1)-2
\end{array}\right]:=\left[\begin{array}{cc}
\alpha_2 & \beta_2\\
\beta_2 & \gamma_2
\end{array}\right].$$

{\bf Subcase 2.2.} $3\leqslant p\leqslant n-k$.

In this subcase, the co-independent pairs are $\{l_1,l_i\}$ ($2\leqslant i\leqslant n-k+1$). Thus, $\sum\limits_{1\leqslant i,j\leqslant n-k+1} b_{ij}=2\cdot\left({n-k+1 \choose 2}-(n-k)\right)=(n-k)(n-k-1)$. So $Q_S=\left[\begin{array}{cc}
(n-k)(n-k+1) & (n-k)^2\\
(n-k)^2 & (n-k)(n-k-1)
\end{array}\right].$ Let's denote 
$$H_3=\left[\begin{array}{cc}
(n-k)(n-k+1) & (n-k)^2\\
(n-k)^2 & (n-k)(n-k-1)
\end{array}\right]:=\left[\begin{array}{cc}
\alpha_3 & \beta_3\\
\beta_3 & \gamma_3
\end{array}\right].$$

{\bf Case 3.} There are no co-independent vertices in $G-S$.

In this case, $\sum\limits_{1\leqslant i\leqslant n-k+1} a_i=n-k+1$. Combined with the fact that $G-S$ contains at least one edge and $G-S\not\cong K_2\cup (n-k-1)K_1$, it contains at most $n-k-2$ isolated vertices. We distinguish three subcases in what follows.

{\bf Subcase 3.1.} There are $n-k-2$ isolated vertices in $G-S$.

In this subcase, $G-S\cong K_3\cup (n-k-2)K_1$. W.l.o.g., assume that $l_1,l_2,l_3$ induce the $K_3$. Then the co-independent pairs of $G-S$ are $\{l_1,l_2\}$, $\{l_1,l_3\}$ and $\{l_2,l_3\}$ and $\sum\limits_{1\leqslant i,j\leqslant n-k+1} b_{ij}=2\cdot\left({n-k+1 \choose 2}-3\right)=(n-k)(n-k+1)-6$. So $Q_S=\left[\begin{array}{cc}
(n-k)(n-k+1) & (n-k)(n-k+1)\\
(n-k)(n-k+1) & (n-k)(n-k+1)-6
\end{array}\right].$

{\bf Subcase 3.2.} There are $n-k-3$ isolated vertices in $G-S$.

In this subcase, $G-S\cong G^*\cup (n-k-3)K_1$, where $G^*$ is a graph of order 4 containing no isolated vertices and not a star. Since $G^*$ is a graph of order 4, the number of co-independent pairs $G-S$ is equal to the one of edges in the complement of $G^*$, i.e., $|E(\overline{G^*})|$. We can see that $|E(\overline{G^*})|\in\{0,1,2,3,4\}$ (see Fig. \ref{fig:4VertexGraph}), then $\sum\limits_{1\leqslant i,j\leqslant n-k+1} b_{ij}=(n-k)(n-k+1)-8$, $(n-k)(n-k+1)-6$, $(n-k)(n-k+1)-4$, $(n-k)(n-k+1)-2$ or $(n-k)(n-k+1)$. Thus, $Q_S=\left[\begin{array}{cc}
(n-k)(n-k+1) & (n-k)(n-k+1)\\
(n-k)(n-k+1) & (n-k)(n-k+1)-8
\end{array}\right],$ $\left[\begin{array}{cc}
(n-k)(n-k+1) & (n-k)(n-k+1)\\
(n-k)(n-k+1) & (n-k)(n-k+1)-6
\end{array}\right],$ $\left[\begin{array}{cc}
(n-k)(n-k+1) & (n-k)(n-k+1)\\
(n-k)(n-k+1) & (n-k)(n-k+1)-4
\end{array}\right],$
$\left[\begin{array}{cc}
(n-k)(n-k+1) & (n-k)(n-k+1)\\
(n-k)(n-k+1) & (n-k)(n-k+1)-2
\end{array}\right]$\\ or  $\left[\begin{array}{cc}
(n-k)(n-k+1) & (n-k)(n-k+1)\\
(n-k)(n-k+1) & (n-k)(n-k+1)
\end{array}\right].$ 
\begin{figure}[!htbp]
\centering
\scalebox{0.5}[0.5]{\includegraphics{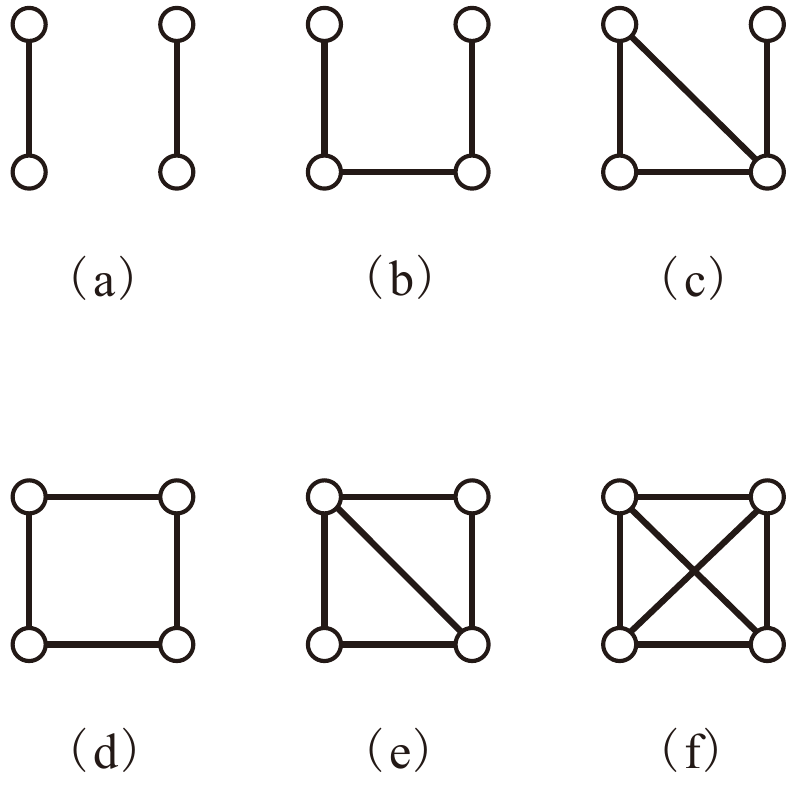}}
\caption{All graphs of order 4 containing no isolated vertices and not a star: (a) contains 4 co-independent pairs, (b) contains 3 co-independent pairs, (c) and (d) contain 2 co-independent pairs, (e) contains 1 co-independent pair and (f) contains no co-independent pairs.}{\label{fig:4VertexGraph}}
\end{figure}

{\bf Subcase 3.3.} There are at most $n-k-4$ isolated vertices in $G-S$.

If $G-S$ contains no co-independent pairs, $Q_S=\left[\begin{array}{cc}
(n-k)(n-k+1) & (n-k)(n-k+1)\\
(n-k)(n-k+1) & (n-k)(n-k+1)
\end{array}\right].$ Now, we assume $G-S$ contains a co-independent pair. W.l.o.g., let $\{l_1,l_2\}$ be a co-independent pair. Since $G-S$ contains no co-independent vertices, $N_{G-S}(l_1)\cap\{l_3,l_4,\cdots,l_{n-k+1}\}\neq\emptyset$ and $N_{G-S}(l_2)\cap\{l_3,l_4,\cdots,l_{n-k+1}\}\neq\emptyset$, where $N_{G-S}(l_i)$ is the neighborhood of $l_i$ in $G-S$. Since there are at most $n-k-4$ isolated vertices in $G-S$, $\left|\left(N_{G-S}(l_1)\cup N_{G-S}(l_2)\right)\cap\{l_3,l_4,\cdots,l_{n-k+1}\}\right|\geqslant 3$. If $\left|N_{G-S}(l_1)\cap\{l_3,l_4,\cdots,l_{n-k+1}\}\right|$ $\geqslant 2$ and $\left|N_{G-S}(l_2)\cap\{l_3,l_4,\cdots,l_{n-k+1}\}\right|\geqslant 2$, then $\{l_1,l_2\}$ is the unique co-independent pair in $G-S$ (see Fig. \ref{fig:Case31}) and further $\sum\limits_{1\leqslant i,j\leqslant n-k+1} b_{ij}=(n-k)(n-k+1)-2$. Otherwise, w.l.o.g., assume $l_2$ has exactly one neighbor $l_3$ in $\{l_3,l_4,\cdots,l_{n-k+1}\}$ in $G-S$. On this assumption, $l_1$ must have two neighbors other than $l_3$ in $\{l_3,l_4,\cdots,l_{n-k+1}\}$. We can see that the co-independent pairs in $G-S$ are $\{l_1,l_2\}$ and $\{l_1,l_3\}$ (see Fig. \ref{fig:Case32}), then $\sum\limits_{1\leqslant i,j\leqslant n-k+1} b_{ij}=(n-k)(n-k+1)-4$. So, in this subcase, $Q_S=\left[\begin{array}{cc}
(n-k)(n-k+1) & (n-k)(n-k+1)\\
(n-k)(n-k+1) & (n-k)(n-k+1)
\end{array}\right],$ $\left[\begin{array}{cc}
(n-k)(n-k+1) & (n-k)(n-k+1)\\
(n-k)(n-k+1) & (n-k)(n-k+1)-2
\end{array}\right]$ or $\left[\begin{array}{cc}
(n-k)(n-k+1) & (n-k)(n-k+1)\\
(n-k)(n-k+1) & (n-k)(n-k+1)-4
\end{array}\right].$

\begin{figure}[!htbp]
\centering
\scalebox{0.5}[0.5]{\includegraphics{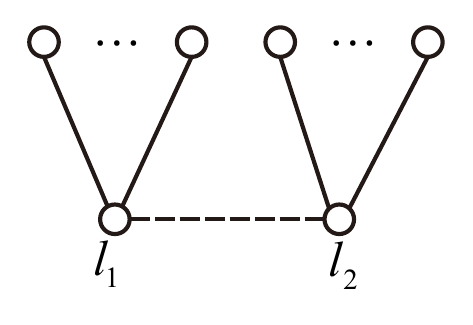}}
\caption{The case that $\left|N_{G-S}(l_1)\cap\{l_3,l_4,\cdots,l_{n-k+1}\}\right|$ $\geqslant 2$ and $\left|N_{G-S}(l_2)\cap\{l_3,l_4,\cdots,\right.$ $\left.l_{n-k+1}\}\right|\geqslant 2$: dashed edge represent that $l_1,l_2$ are adjacent or not.}{\label{fig:Case31}}
\end{figure}

\begin{figure}[!htbp]
\centering
\scalebox{0.5}[0.5]{\includegraphics{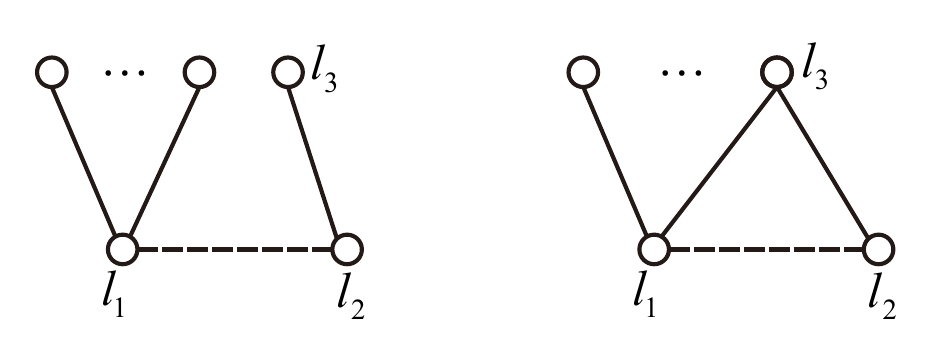}}
\caption{The case that $l_2$ has exactly one neighbor $l_3$ in $\{l_3,l_4,\cdots,l_{n-k+1}\}$: dashed edge represent that $l_1,l_2$ are adjacent or not.}{\label{fig:Case32}}
\end{figure}

Let's denote $$H_4=\left[\begin{array}{cc}
(n-k)(n-k+1) & (n-k)(n-k+1)\\
(n-k)(n-k+1) & (n-k)(n-k+1)-8
\end{array}\right]:=\left[\begin{array}{cc}
\alpha_4 & \beta_4\\
\beta_4 & \gamma_4
\end{array}\right],$$
$$H_5=\left[\begin{array}{cc}
(n-k)(n-k+1) & (n-k)(n-k+1)\\
(n-k)(n-k+1) & (n-k)(n-k+1)-6
\end{array}\right]:=\left[\begin{array}{cc}
\alpha_5 & \beta_5\\
\beta_5 & \gamma_5
\end{array}\right],$$
$$H_6=\left[\begin{array}{cc}
(n-k)(n-k+1) & (n-k)(n-k+1)\\
(n-k)(n-k+1) & (n-k)(n-k+1)-4
\end{array}\right]:=\left[\begin{array}{cc}
\alpha_6 & \beta_6\\
\beta_6 & \gamma_6
\end{array}\right],$$
$$H_7=\left[\begin{array}{cc}
(n-k)(n-k+1) & (n-k)(n-k+1)\\
(n-k)(n-k+1) & (n-k)(n-k+1)-2
\end{array}\right]:=\left[\begin{array}{cc}
\alpha_7 & \beta_7\\
\beta_7 & \gamma_7
\end{array}\right]$$
and
$$H_8=\left[\begin{array}{cc}
(n-k)(n-k+1) & (n-k)(n-k+1)\\
(n-k)(n-k+1) & (n-k)(n-k+1)
\end{array}\right]:=\left[\begin{array}{cc}
\alpha_8 & \beta_8\\
\beta_8 & \gamma_8
\end{array}\right].$$ 

Now, it has been proved that, for any co-dependent set $S\in\overline{\mathcal{D}}_{k-1}(G)$, the matrix $Q_S=H_i$ for some $i\in\{1,2,\cdots,8\}$. Denote $t_i:=|\{S\in\overline{\mathcal{D}}_{k-1}(G)|Q_S=H_i\}|$ ($i=1,2,\cdots,8$). By (\ref{eq:sum}), we obtain:
\begin{equation}\label{eq:sum2}
\det\left(\nabla^2q_k\right)=\left((n-k-1)!\right)^2\det\left(\sum_{i=1}^8t_iH_i\right)=\left((n-k-1)!\right)^2\left|\begin{array}{cc}
\sum\limits_{i=1}^8\alpha_it_i & \sum\limits_{i=1}^8\beta_it_i\\
\sum\limits_{i=1}^8\beta_it_i & \sum\limits_{i=1}^8\gamma_it_i
\end{array}\right|.
\end{equation}
The determinant $\left|\begin{array}{cc}
\sum\limits_{i=1}^8\alpha_it_i & \sum\limits_{i=1}^8\beta_it_i\\
\sum\limits_{i=1}^8\beta_it_i & \sum\limits_{i=1}^8\gamma_it_i
\end{array}\right|$ is a quadratic form in $t_1,t_2,\cdots,t_8$. Let's denote its matrix representation: 
\begin{equation}\label{eq:MatrixRepresentation}
\left|\begin{array}{cc}
\sum\limits_{i=1}^8\alpha_it_i & \sum\limits_{i=1}^8\beta_it_i\\
\sum\limits_{i=1}^8\beta_it_i & \sum\limits_{i=1}^8\gamma_it_i
\end{array}\right|:=\mathbf{t}^{\mathsf{T}}A\mathbf{t},
\end{equation}
where $\mathbf{t}=(t_1,t_2,t_3,t_4,t_5,t_6,t_7,t_8)^{\mathsf{T}}$, $A$ is an $8\times 8$ symmetric matrix and $A_{ij}=\frac{1}{2}(\alpha_i\gamma_j+\alpha_j\gamma_i-2\beta_i\beta_j)$ for $1\leqslant i,j\leqslant 8$. For convenience, we denote $r:=n-k+1$. $r\geqslant 4$ since $k\leqslant n-3$. By the calculation, we obtain: 
\begin{equation}\label{eq:MatrixA}
A=(r-1)\cdot\left[\begin{array}{cccccccc}
-2(r-2) & -2(r-1) & -(r-2) & -3r & -2r & -r & 0 & r\\
-2(r-1) & -(3r-1) & -(2r-1) & -5r & -4r & -3r & -2r & -r\\
-(r-2) & -(2r-1) & -(r-1) & -4r & -3r & -2r & -r & 0\\
-3r & -5r & -4r & -8r & -7r & -6r & -5r & -4r\\
-2r & -4r & -3r & -7r & -6r & -5r & -4r & -3r\\
-r & -3r & -2r & -6r & -5r & -4r & -3r & -2r\\
0 & -2r & -r & -5r & -4r & -3r & -2r & -r\\
r & -r & 0  & -4r & -3r & -2r & -r & 0\\
\end{array}\right].
\end{equation}

If $G$ is $K_2\cup 2K_1$-free, for any co-dependent set $S$ of size $k-1$, $G-S\not\cong K_2\cup (n-k-1)K_1$, that is, $t_1=0$. If $G$ has an independent set of size $n-2$, for any co-dependent set $S$ of size $k-1$, $G-S$ has an independent set of size $n-k-1$. Then $G-S$ must contain co-independent pairs. It deduces that $t_8=0$. 

Since $t_i\geqslant 0$ for all $1\leqslant i\leqslant 8$, combined with (\ref{eq:sum2}),  (\ref{eq:MatrixRepresentation}) and (\ref{eq:MatrixA}), we can see that $\det(\nabla^2q_k)\leqslant 0$ if $t_1=0$ or $t_8=0$. This completes the proof.
\end{proof}

Now, we prove Theorem \ref{thm:ULC}.
\begin{proof}[Proof of Theorem \ref{thm:ULC}]
By (\ref{eq:qk}), 
\begin{eqnarray*}
\det(\nabla^2q_k)&=&\left|\begin{array}{cc}
(n-k+1)!\bar{d}_{k-1} & (n-k)!k\bar{d}_k\\
(n-k)!k\bar{d}_k & (n-k-1)!k(k+1)\bar{d}_{k+1}
\end{array}\right|\\
 &=&k^2\left((n-k)!\right)^2\left(\left(1+\frac{1}{k}\right)\left(1+\frac{1}{n-k}\right)\bar{d}_{k-1}\bar{d}_{k+1}-\bar{d}_k^2\right).
\end{eqnarray*}

By Lemma \ref{lem:Hessian<=0}, if $G$ is $K_2\cup 2K_1$-free or has an independent set of size $n-2$, $\det(\nabla^2q_k)\leqslant 0$ for $1\leqslant k\leqslant n-3$, that is, $\left(1+\frac{1}{k}\right)\left(1+\frac{1}{n-k}\right)\bar{d}_{k-1}\bar{d}_{k+1}-\bar{d}_k^2\leqslant 0$ for $1\leqslant k\leqslant n-3$. Combined with $\bar{d}_n=\bar{d}_{n-1}=0$, we obtain that $\{\bar{d}_k\}_{k=0}^n$ is ultra log-concave. Thus, $D(G,x)$ is ultra log-concave.
\end{proof}
\section{Proof of Theorem \ref{thm:RR}}
Before the proof of Theorem \ref{thm:RR}, we will introduce some lemmas and terminology in what follows. Firstly, we have the following lemma:
\begin{lem}\label{lem:IsolatedVertices}
For any graph $G$ and $m\in\mathds{N}$, $D(G\cup mK_1,x)=(x+1)^mD(G,x)$.
\end{lem}
\begin{proof}
Assume that $|V(G)|=n$. It is known that $I(G\cup mK_1,x)=(I(K_1,x))^mI(G,x)=(x+1)^mI(G,x)$ (see \cite{gh83}). Then, by Proposition \ref{pro:Complementary}, $D(G\cup mK_1,x)=(x+1)^{m+n}-I(G\cup mK_1,x)=(x+1)^{m+n}-(x+1)^mI(G,x)=(x+1)^m[(x+1)^n-I(G,x)]=(x+1)^mD(G,x)$.
\end{proof}

Let $P(x)=\sum\limits_{k=0}^ma_kx^k$ be a polynomial. $P'(x):=\sum\limits_{k=0}^{m-1}(k+1)a_{k+1}x^k$ denotes the {\em derivative} of $P(x)$. It is known that
\begin{pro}[Folklore]\label{pro:Derivative}
Let $P(x)$ be a real polynomial with degree at least 2. If $P(x)$ is real-rooted, so is $P'(x)$.
\end{pro}

The following proposition will be useful in our proof. 
\begin{lem}\label{lem:3ConsecutiveCoeffiicients}
Let $P(x)=\sum\limits_{k=0}^ma_kx^k$ be a real polynomial with $a_k\geqslant 0$ for $1\leqslant k\leqslant m$. If there exists an integer $j$ for $1\leqslant j\leqslant m-1$ such that $a_{j-1}=a_j=a_{j+1}>0$, then $P(x)$ is not real-rooted.
\end{lem}
\begin{proof}
{\bf Case 1.} $a_0>0$.

In this case, if there exists an integer $j$ such that $a_{j-1}=a_j=a_{j+1}>0$ for $1\leqslant j\leqslant m-1$, then
$$a_j^2<\left(1+\frac{1}{j}\right) a_{j}^2=\left(1+\frac{1}{j}\right) a_{j-1}a_{j+1}.$$
That is, $P(x)$ is not ordered log-concave, so it is not real-rooted.

{\bf Case 2.} $a_0\leqslant 0$.

In this case, if there exists an integer $j$ such that $a_{j-1}=a_j=a_{j+1}>0$, then $2\leqslant j\leqslant m-1$. By contradiction assume $P(x)$ is real-rooted. Then $P'(x)$ is real-rooted by Proposition \ref{pro:Derivative}. The coefficient sequence of $P'(x)$, $\left\{(k+1)a_{k+1}\right\}^{m-1}_{k=0}$, is a non-negative sequence. Let's consider the coefficients of the terms $x^j$, $x^{j-1}$ and $x^{j-2}$ of $P'(x)$. We obtain
\begin{equation*}
\left(1+\frac{1}{j-1}\right)(j+1)a_{j+1}\cdot (j-1)a_{j-1}=j(j+1)a^2_j>(ja_j)^2.
\end{equation*}
Thus, $P'(x)$ is not ordered log-concave, a contradiction with real-rootedness of $P'(x)$.
\end{proof}

Now, we introduce a necessary and sufficient condition of real-rootedness of real polynomials: Sturm's Theorem. Firstly, we give the definition of Sturm sequence: 
\begin{den}\cite{b03,bh02,mr23}\label{def:SturmSequence}
Let $P(x)$ be a real polynomial with degree at least 1 and positive leading coefficient. The polynomial sequence 
\begin{equation*}
P_0(x),P_1(x),\cdots,P_k(x)
\end{equation*}
is called a {\em Sturm sequence} of $P(x)$ if it satisfies
\begin{enumerate}
\item $P_0(x):=P(x)$,
\item $P_1(x):=P'(x)$, 
\item $P_i(x):=-R_{i-2,i-1}(x)$ for $i\geqslant 2$, where
\begin{equation*}
P_{i-2}(x)=Q_{i}(x)P_{i-1}(x)+R_{i-2,i-1}(x)
\end{equation*}
and the degree of $R_{i-2,i-1}(x)$ is less than $P_{i-1}(x)$, i.e., $R_{i-2,i-1}(x)$ is the remainder of the division of $P_{i-2}(x)$ by $P_{i-1}(x)$, and
\item $P_k(x)$ is the polynomial with the smallest possible non-negative degree such that $P_k(x)\neq 0$.
\end{enumerate}
\end{den}
Sturm's Theorem is shown in what follows: 
\begin{lem}[Sturm's Theorem]{\em\cite{b03,bh02,mr23}}\label{lem:SturmsTheorem}
Let $P(x)$ be a real polynomial with degree at least 1 and positive leading coefficient. Let $P_0(x),P_1(x),\cdots,P_k(x)$ be Sturm sequence of $P(x)$, where the degree of $P_i(x)$ is $d_i$ for $0\leqslant i\leqslant k$. $P(x)$ is real-rooted if and only if $P_i(x)$ has a positive leading coefficient for $0\leqslant i\leqslant k$ and $d_i-d_{i+1}=1$ for $0\leqslant i\leqslant k-1$.  
\end{lem}
By using Lemma \ref{lem:SturmsTheorem}, we obtain the following lemma:
\begin{lem}\label{lem:Degree>=n-2}
Let $G$ be a graph of order $n$ and $\mathcal{A}$ a graph property. If $P_{\mathcal{A}}(G,x)$ is real-rooted and $P_{\mathcal{A}}(G,x)\neq (x+1)^n$, then there exists a subset $S$ of $V(G)$ such that $|S|\geqslant n-2$ and $G[S]\not\in\mathcal{A}$. In particular, if $G$ contains at least one edge and $D(G,x)$ is real-rooted, then $G$ has an independent set of size $n-2$.
\end{lem}
\begin{proof}
By contradiction assume that $G[S]\in\mathcal{A}$ for any $S\subseteq V(G)$ where $|S|\geqslant n-2$. By the definition, the degree of $P_{\bar{\mathcal{A}}}(G,x)$ is at most $n-3$. Combined with $P_{\mathcal{A}}(G,x)\neq (x+1)^n$ and Proposition \ref{pro:Complementary}, $P_{\bar{\mathcal{A}}}(G,x)\neq 0$ or $(x+1)^n$. Let $P_0(x),P_1(x),\cdots,P_k(x)$ be Sturm sequence of $P_{\mathcal{A}}(G,x)$, where the degree of $P_i(x)$ is $d_i$ for $0\leqslant i\leqslant k$. By the definition of Sturm sequence, $P_0(x):=P_{\mathcal{A}}(G,x)=(x+1)^n-P_{\bar{\mathcal{A}}}(G,x)$, $P_1(x):=P'_{\mathcal{A}}(G,x)=n(x+1)^{n-1}-P'_{\bar{\mathcal{A}}}(G,x)$. Since the degree of $P_{\bar{\mathcal{A}}}(G,x)$ is at most $n-3$, the degree of $P'_{\bar{\mathcal{A}}}(G,x)$ is at most $n-4$ and $d_1=n-1$. Since $P_{\bar{\mathcal{A}}}(G,x)\neq 0$ or $(x+1)^n$, $P_2(x)=\frac{1}{n}(x+1)P_1(x)-P_0(x)=P_{\bar{\mathcal{A}}}(G,x)-\frac{1}{n}(x+1)P'_{\bar{\mathcal{A}}}(G,x)\neq 0$. We can see that $d_2\leqslant n-3$, then $d_1-d_2\geqslant 2$, a contradiction with Lemma \ref{lem:SturmsTheorem}.
\end{proof}

Now, we can prove Theorem \ref{thm:RR}.
\begin{proof}[Proof of Theorem \ref{thm:RR}]
{\em Sufficiency.} By Lemma \ref{lem:IsolatedVertices}, we only need to prove the real-rootedness of the dependence polynomials of $K_2$, $P_3$, $K_3$, $2K_2$, $P_4$, $C_4$, $G_1$ and $G_2$. Let's calculate them one by one: $D(K_2,x)=x^2$; $D(P_3,x)=x^3+2x^2=x^2(x+2)$; $D(K_3,x)=x^3+3x^2=x^2(x+3)$; $D(2K_2,x)=x^4+4x^3+2x^2=x^2(x+2-\sqrt{2})(x+2+\sqrt{2})$; $D(P_4,x)=x^4+4x^3+3x^2=x^2(x+1)(x+3)$; $D(C_4,x)=x^4+4x^3+4x^2=x^2(x+2)^2$; $D(G_1,x)=x^4+4x^3+4x^2=x^2(x+2)^2$;  $D(G_2,x)=x^5+5x^4+8x^3+4x^2=x^2(x+1)(x+2)^2$. They are all real-rooted.

{\em Necessity.} Let $G$ be a graph of order $n$ which contains at least one edge. Assume that $D(G,x)$ is real-rooted. By Lemma \ref{lem:Degree>=n-2}, $G$ has an independent set of size $n-2$, denoted by $X$. Set $\{u,v\}:=V(G)\setminus X$. Let's denote $A:=N(u)\cap X$, $B:=N(v)\cap X$ 
and $C:=A\cup B$. Set $a:=|A|$, $b:=|B|$, $c:=|C|$. W.l.o.g., we assume $a\leqslant b$. If $u,v$ are not adjacent, we denote $G$ by $G_{a,b,c}$. Otherwise, we denote it by $G^+_{a,b,c}$ (see Fig. \ref{fig:Guv}).
\begin{figure}[!htbp]
\centering
\scalebox{0.5}[0.5]{\includegraphics{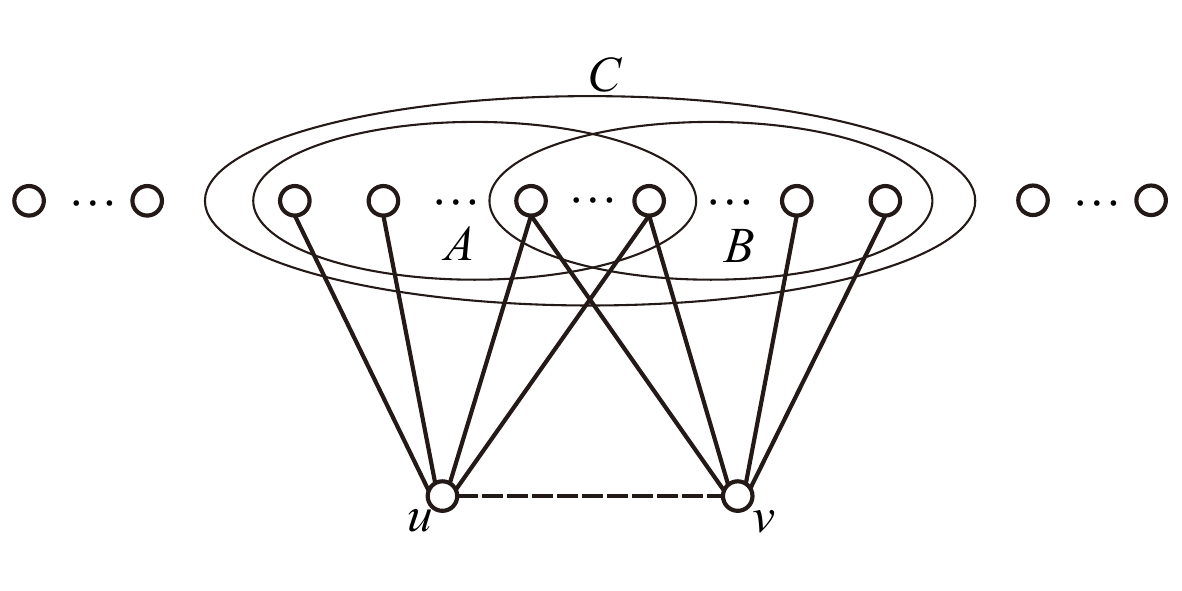}}
\caption{$G^+_{a,b,c}$ or $G_{a,b,c}$: dashed edge represent that $u,v$ are adjacent or not.}{\label{fig:Guv}}
\end{figure}

Now, let's calculate $D(G_{a,b,c},x)$ and $D(G^+_{a,b,c},x)$. For $G\in\{G_{a,b,c},G^+_{a,b,c}\}$, let $d^{(1)}_k(G)$ (resp. $d^{(2)}_k(G)$, or $d^{(3)}_k(G)$) be the number of dependence sets of size $k$ which contains $u$ but not $v$ (resp. $v$ but not $u$, or both $u$ and $v$). Denote $D_1(G,x):=\sum\limits_{k\geqslant 0}d^{(1)}_k(G)x^k$ (resp. $D_2(G,x):=\sum\limits_{k\geqslant 0}d^{(2)}_k(G)x^k$, $D_3(G,x):=\sum\limits_{k\geqslant 0}d^{(3)}_k(G)x^k$). Since $X$ is an independent set, a dependent set of $G$ must contain $u$ or $v$. Thus, $D(G,x)=D_1(G,x)+D_2(G,x)+D_3(G,x)$. Now, we calculate $D_1(G,x)$, $D_2(G,x)$ and $D_3(G,x)$. 
The Product formula of generating function in what follows will be used. 

{\bf The Product formula.} (see \cite{b11}, p. 156, Theorem 8.5) Let $k\geqslant 2$ be an integer. Let $a^{(1)}_m$ be the number of ways to build a certain structure on an $m$-element set, $a^{(2)}_m$ the number of  ways to build the second structure on an $m$-element set, $\cdots$, $a^{(k)}_m$ the number of ways to build the $k$-th structure on an $m$-element set. Let $a_m$ be the number of ways to separate an $m$-element set into the $k$ disjoint subsets $S_1,S_2,\cdots,S_k$ 
($S_1,S_2,\cdots,S_k$ are allowed to be empty), and then to build a structure of the first kind on $S_1$, a structure of the second kind on $S_2$, $\cdots$, a structure of the $k$-th kind on $S_k$. Let $P_1(x),P_2(x),\cdots,P_k(x),P(x)$ be the respective generating functions of the sequences $\{a^{(1)}_m\}$, $\{a^{(2)}_m\}$, $\cdots$, $\{a^{(k)}_m\}$, $\{a_m\}$. Then $P(x)=P_1(x)P_2(x)\cdots P_k(x)$. 

For $G\in\{G_{a,b,c},G^+_{a,b,c}\}$, a dependent set $S$ containing $u$ but not $v$ can be separated into three disjoint parts: $\{u\}$, $S\cap A$ and $S\cap (X\setminus A)$. Firstly, since $u$ must be chosen in $S$, the generating function for the sequence of numbers of ways to choose vertices in $\{u\}$ is $x$. Then, to ensure that $G[S]$ contains at least one edge, $S\cap A\neq\emptyset$, so the generating function for the sequence of numbers of ways to choose vertices in $A$ is $\sum\limits_{k=1}^a{a\choose k}x^k=(x+1)^a-1$. Finally, the generating function for the sequence of numbers of ways to choose vertices in $X\setminus A$ is $\sum\limits_{k=0}^{n-a-2}{n-a-2\choose k}x^k=(x+1)^{n-a-2}$. By the Product formula, $D_1(G,x)=x(x+1)^{n-a-2}[(x+1)^a-1]$. Similarly, $D_2(G,x)=x(x+1)^{n-b-2}[(x+1)^b-1]$.

Now let's consider $D_3(G,x)$. When $G=G_{a,b,c}$, for every dependent set $S$ containing both $u$ and $v$, $S\cap C\neq\emptyset$ since $\{u,v\}$ is an independent set and there must exist at least one vertex in $S\cap C$ to ensure that $G[S]$ contains edges. Similar to the analysis of $D_1(G,x)$, we obtain that $D_3(G_{a,b,c},x)=x^2(x+1)^{n-c-2}[(x+1)^c-1]$. When $G=G^+_{a,b,c}$, since $\{u,v\}$ is a dependent set, the union of $\{u,v\}$ and every subset of $X$ is also a dependent set. Thus, $D_3(G^+_{a,b,c},x)=x^2(x+1)^{n-2}$.

Then 
\begin{eqnarray*}
D(G_{a,b,c},x)&=&D_1(G_{a,b,c},x)+D_2(G_{a,b,c},x)+D_3(G_{a,b,c},x)\\
 &=&x(x+1)^{n-a-2}[(x+1)^a-1]+x(x+1)^{n-b-2}[(x+1)^b-1]\\
 & &{}+x^2(x+1)^{n-c-2}[(x+1)^c-1];\\
D(G^+_{a,b,c},x)&=&D_1(G^+_{a,b,c},x)+D_2(G^+_{a,b,c},x)+D_3(G^+_{a,b,c},x)\\
 &=&x(x+1)^{n-a-2}[(x+1)^a-1]+x(x+1)^{n-b-2}[(x+1)^b-1]+x^2(x+1)^{n-2}.
\end{eqnarray*}

It is obvious that a polynomial $P(x)$ is real-rooted if and only if $P(x-1)$ is real-rooted. Let's consider $D(G_{a,b,c},x-1)$ and $D(G^+_{a,b,c},x-1)$. Noting that $0\leqslant a\leqslant b\leqslant c$, we have
\begin{eqnarray*}
D(G_{a,b,c},x-1)&=&(x-1)x^{n-c-2}(x^{c+1}+x^c-x^{c-a}-x^{c-b}-x+1);\\
D(G^+_{a,b,c},x-1)&=&(x-1)x^{n-b-2}(x^{b+1}+x^b-x^{b-a}-1).
\end{eqnarray*}

Let's denote $Q_1(x):=x^{c+1}+x^c-x^{c-a}-x^{c-b}-x+1$ and $Q_2(x):=x^{b+1}+x^b-x^{b-a}-1$. Then $D(G_{a,b,c},x)$ is real-rooted if and only if $Q_1(x)$ is real-rooted; $D(G^+_{a,b,c},x)$ is real-rooted if and only if $Q_2(x)$ is real-rooted. Now, we discuss the real-rootedness of $Q_1(x)$ and $Q_2(x)$.

{\bf Case 1.} The real-rootedness of $Q_1(x)$.

We distinguish two subcases: $b=c$ or not.

{\bf Subcase 1.1.} $b=c$.

When $a=b=c$, $Q_1(x)=x^{c+1}+x^c-x-1=(x+1)(x^c-1)$. It is real-rooted if and only if $c\leqslant 2$. Since $G_{a,b,c}$ contains at least one edge, $c\neq 0$. When $a=b=c=1$, $G_{a,b,c}\cong P_3\cup (n-3)K_1$. When $a=b=c=2$, $G_{a,b,c}\cong C_4\cup (n-4)K_1$.

When $a<b=c$, $Q_1(x)=x^{c+1}+x^c-x^{c-a}-x$. 
Since $Q_1(0)=Q_1(1)=0$, we can assume $Q_1(x):=x(x-1)Q_3(x)$. Then $Q_1(x)$ is real-rooted if and only if either $Q_3(x)$ is real-rooted, or $Q_3(x)\neq 0$ is a constant function. 
When $a=0$, $Q_3(x)=x^{c-1}+x^{c-2}+\cdots+x+1$. If it is real-rooted, then $c=2$ and further $G_{a,b,c}\cong P_3\cup (n-3)K_1$; if it is a non-zero constant function, then $c=1$ and further $G_{a,b,c}\cong K_2\cup (n-2)K_1$. 
When $a\geqslant 1$, $Q_3(x)=x^{c-1}+2x^{c-2}+2x^{c-3}+\cdots+2x^{c-a-1}+x^{c-a-2}+x^{c-a-3}+\cdots+x+1$. If $Q_3(x)$ is a non-zero constant function, then $c=1$ and it contradicts with $a<c$. Now assume that $Q_3(x)$ is real-rooted. By Lemma \ref{lem:3ConsecutiveCoeffiicients}, $c-3\leqslant c-a-1$, i.e., $a\leqslant 2$. If $c-a-3\geqslant 0$, considering the coefficients of the terms $x^{c-a-1}$, $x^{c-a-2}$, $x^{c-a-3}$ of $Q_3(x)$, we obtain $2\cdot 1>1^2$. Therefore, $Q_3(x)$ is not log-concave, a contradiction. Thus, $c\leqslant a+2$.  When $a=1$, $b=c=2$, $G_{a,b,c}\cong P_4\cup (n-4)K_1$. When $a=1$, $b=c=3$, $G_{a,b,c}\cong G_2\cup (n-5)K_1$. When $a=2$, $b=c=3$,  $Q_3(x)=x^2+2x+2$, which is not real-rooted, a contradiction. When $a=2$, $b=c=4$, $Q_3(x)=x^3+2x^2+2x+1=(x+1)(x^2+x+1)$, which is not real-rooted, also a contradiction.

{\bf Subcase 1.2.} $b<c$.

In this subcase, since $b<c\leqslant a+b$, it can be deduced that $a\geqslant 1$. We can obtain that $Q_1(x)=x^{c+1}+x^c-x^{c-a}-x^{c-b}-x+1=(x-1)(x^c+2x^{c-1}+2x^{c-2}+\cdots+2x^{c-a}+x^{c-a-1}+x^{c-a-2}+\cdots+x^{c-b}-1)$. Let's denote $Q_4(x):=x^c+2x^{c-1}+2x^{c-2}+\cdots+2x^{c-a}+x^{c-a-1}+x^{c-a-2}+\cdots+x^{c-b}-1$. Since $c\neq 0$, $Q_1(x)$ is real-rooted if and only if $Q_4(x)$ is real-rooted. Assume $Q_4(x)$ is real-rooted. Similar to Subcase 1.1, by Lemma \ref{lem:3ConsecutiveCoeffiicients}, $c-2\leqslant c-a$, i.e., $a\leqslant 2$. Now, we prove that $b\leqslant a+1$. By contradiction assume $b\geqslant a+2$, then $c-a-3\geqslant 0$ since $c>b$. The derivative $Q'_4(x)=cx^{c-1}+2(c-1)x^{c-2}+\cdots+2(c-a)x^{c-a-1}+(c-a-1)x^{c-a-2}+(c-a-2)x^{c-a-3}+\cdots+(c-b)x^{c-b-1}$. By Proposition \ref{pro:Derivative}, $Q'_4(x)$ is real-rooted. Considering the coefficients of the terms $x^{c-a-1}$, $x^{c-a-2}$, $x^{c-a-3}$ of $Q'_4(x)$, we obtain $2(c-a)(c-a-2)=2(c-a-1)^2-2\geqslant (c-a-1)^2+2>(c-a-1)^2$. Therefore, $Q'_4(x)$ is not log-concave, a contradiction. Thus, $b\leqslant a+1$. When $a=1$, $b=1$, $c=2$, $G_{a,b,c}\cong 2K_2\cup (n-4)K_1$. When $a=1$, $b=2$, $c=3$, $Q_4(x)=x^3+2x^2+x-1$. By analyzing the cubic function, $Q_4(x)$ has only one real root, a contradiction. When $a=2$, $b=2$, $c=3$, $Q_4(x)=x^3+2x^2+2x-1$. $Q'_4(x)=3x^2+4x+2$, which is not real-rooted, a contradiction with Proposition \ref{pro:Derivative}. When $a=2$, $b=2$, $c=4$, $Q_4(x)=x^4+2x^3+2x^2-1=(x+1)(x^3+x^2+x-1)$. $(x^3+x^2+x-1)'=3x^2+2x+1$ is not real-rooted, neither is $x^3+x^2+x-1$, a contradiction with the real-rootedness of $Q_4(x)$. When $a=2$, $b=3$, $c=4$, $Q_4(x)=x^4+2x^3+2x^2+x-1$. $Q'_4(x)=4x^3+6x^2+4x+1=(2x+1)(2x^2+2x+1)$, which is not real-rooted, a contradiction with Proposition \ref{pro:Derivative}. When $a=2$, $b=3$, $c=5$, $Q_4(x)=x^5+2x^4+2x^3+x^2-1$ and $Q'_4(x)=5x^4+8x^3+6x^2+2x=x(5x^3+8x^2+6x+2)$. Since $6^2<\left(1+\frac{1}{1}\right)\left(1+\frac{1}{3-1}\right)\cdot 2\cdot 8$, the polynomial $5x^3+8x^2+6x+2$ is not real-rooted, neither is $Q'_4(x)$, a contradiction with Proposition \ref{pro:Derivative}. 

{\bf Case 2.} The real-rootedness of $Q_2(x)$.

We distinguish two subcases: $a=0$ or $a\geqslant 1$.

{\bf Subcase 2.1.} $a=0$.

In this subcase, $Q_2(x)=x^{b+1}-1$. It is real-rooted if and only if $b=0$ or 1. When $b=0$, $G^+_{a,b,c}\cong K_2\cup (n-2)K_1$. When $b=1$, $G^+_{a,b,c}\cong P_3\cup (n-3)K_1$. 

{\bf Subcase 2.2.} $a\geqslant 1$.

In this subcase, $Q_2(x)=x^{b+1}+x^b-x^{b-a}-1=(x-1)(x^b+2x^{b-1}+2x^{b-2}+\cdots+2x^{b-a}+x^{b-a-1}+x^{b-a-2}+\cdots+x+1)$. Let's denote $Q_5(x):=x^b+2x^{b-1}+2x^{b-2}+\cdots+2x^{b-a}+x^{b-a-1}+x^{b-a-2}+\cdots+x+1$. Since $b\geqslant a\geqslant 1$, $Q_2(x)$ is real-rooted if and only if $Q_5(x)$ is real-rooted. Assume $Q_5(x)$ is real-rooted. Similarly, we can deduce that $a\leqslant 2$ and $b-a=0$ or 1. When $a=1$, $b=1$, $G^+_{a,b,c}\cong K_3 \cup (n-3)K_1$ if $c=1$, and $G^+_{a,b,c}\cong P_4 \cup (n-4)K_1$ if $c=2$. When $a=1$, $b=2$, $G^+_{a,b,c}\cong G_1 \cup (n-4)K_1$ if $c=2$, and $G^+_{a,b,c}\cong G_2 \cup (n-5)K_1$ if $c=3$. When $a=2$, $b=2$, $Q_5(x)=x^2+2x+2$, which is not real-rooted, a contradiction. When $a=2$, $b=3$, $Q_5(x)=x^3+2x^2+2x+1=(x+1)(x^2+x+1)$, which is not real-rooted, a contradiction.

In conclusion, the proof is completed.
\end{proof}

\section{Conclusions and future work}
In the present paper, we proved that, for a graph $G$ of order $n$, the dependence polynomial $D(G,x)$ is ultra log-concave if $G$ is a $K_2\cup 2K_1$-free graph or has an independent set of size $n-2$. And further, the graphs of order $n$ with real-rooted dependence polynomials are characterized, which is exactly one of the following eight graphs: $K_2\cup (n-2)K_1$, $P_3\cup (n-3)K_1$, $K_3\cup (n-3)K_1$, $2K_2\cup (n-4)K_1$, $P_4\cup (n-4)K_1$, $C_4\cup (n-4)K_1$, $G_1\cup (n-4)K_1$, $G_2\cup (n-5)K_1$. 

Combined with the proof of Theorem \ref{thm:ULC}, we obtain that, for a graph $G$, $D(G,x)$ is ultra log-concave if and only if $\det(\nabla^2q_k)\leqslant 0$. Observing (\ref{eq:sum2}), (\ref{eq:MatrixRepresentation}) and (\ref{eq:MatrixA}), we notice that all the coefficients of the quadratic form (\ref{eq:MatrixRepresentation}) is nonpositive except the one of $t_1t_8$. Thus, if $\det(\nabla^2q_k)>0$, $t_i$ ($i=2,3,\cdots,7$) must be much less than $t_1$ and $t_8$. We believe that there are no graphs satisfying such condition, so we conjecture:
\begin{con}\label{con:ULCofDependencePolynomials}
For every graph $G$, $D(G,x)$ is ultra log-concave.
\end{con}

For the graph polynomials associated with co-hereditary graph properties, Makowsky and Rakita \cite{mr23} showed that these polynomials are unimodal for almost all graphs. Meanwhile, they asked: Under what conditions can almost unimodality be improved to unimodality or log-concavity?

Let $G$ be a graph with $V(G)=[n]$ and $\mathcal{A}$ a co-hereditary graph property. Let $\mathcal{I}$ denote the family of sets of vertices $S$ satisfying $G-S\in\mathcal{A}$, i.e., $\mathcal{I}:=\{S\subseteq V(G)|G-S\in\mathcal{A}\}$. Since $\mathcal{A}$ is co-hereditary, it is easy to see that $([n],\mathcal{I})$ is an independence system. Let's denote $\mathcal{A}_{\min}:=\{H\in\mathcal{A}|\mbox{any induced subgraph }H'\neq H\mbox{ of }H\mbox{ is not in }\mathcal{A}\}$. Assume $\max\limits_{H\in\mathcal{A}_{\min}}|V(H)|$ is finite and denote $l:=\max\limits_{H\in\mathcal{A}_{\min}}|V(H)|$. We have
\begin{pro}\label{pro:ExchangeProperty}
$G$, $\mathcal{A}$, $\mathcal{I}$ and $l$ are defined above. Then the independence system $([n],\mathcal{I})$ satisfies the following condition:
\begin{equation}
\text{For any }S,T\in\mathcal{I}\text{ where }|T|\geqslant |S|+l\text{, there exists }x\in T\setminus S\text{ such that }S\cup\{x\}\in\mathcal{I}. \tag{$\ast$}
\end{equation}
\end{pro}
\begin{proof}
If $S\subset T$, since $T$ is an independence set, $S\cup\{x\}\in\mathcal{I}$ for any $x\in T\setminus S$. Otherwise, $|T\setminus S|\geqslant l+1$, so $|S|\leqslant |S\cup T|-(l+1)\leqslant |V(G)|-(l+1)$. Since $S\in\mathcal{I}$, by the definition, $G-S\in\mathcal{A}$, that is, there is an induced subgraph of $G-S$ in $\mathcal{A}_{\min}$, denoted by $H$. By the definition of $l$, $|V(H)|\leqslant l$. Recall that $|T\setminus S|\geqslant l+1$, so there must exist a vertex $x\in T\setminus S$ and $x\not\in V(H)$. For such $x$, $H$ is still an induced subgraph of $G-(S\cup\{x\})$, so $S\cup\{x\}\in\mathcal{I}$. 
\end{proof}

More abstractly, for $l\in\mathds{N}^*$, we call an independence system $\mathcal{M}:=([n],\mathcal{I})$ satisfying ($*$) an {\em $l$-matroid}. By this definition, we can see that if $\mathcal{M}$ is an $l_1$-matroid, $\mathcal{M}$ is also an $l_2$-matroid for all $l_2>l_1$. In particular, 1-matroids are commonly referred to as {\em matroids}. For matroids, the following famous conjecture was posed by Mason:
\begin{con}[Mason's Conjecture]\label{con:MasonConjecture}{\em\cite{m72}}
Let $\mathcal{M}$ be an $n$-element matroid and $i_k$ the number of independent sets of size $k$. Then
\begin{enumerate}
\renewcommand{\labelenumi}{(\roman{enumi})}
\item the sequence $\{i_k\}_{k=0}^n$ is log-concave;
\item the sequence $\{i_k\}_{k=0}^n$ is ordered log-concave;
\item the sequence $\{i_k\}_{k=0}^n$ is ultra log-concave.
\end{enumerate}
\end{con}

Adiprasito, Huh, and Katz \cite{ahk18} proved (i) using techniques from Hodge theory and algebraic geometry in 2018. (ii) was proved by Huh, Schr\"oter and Wang \cite{hsw22} in 2022. More recently, (iii) was proved by Anari et al. \cite{algv18}, Br\"and\'en and Huh \cite{bh18}, and Chan and Pak \cite{cp22} in different methods independently. Motivated by the discussions above, our question is: under what conditions can the analogous version of Mason's Conjecture hold true for $l$-matroids? By Proposition \ref{pro:ExchangeProperty}, it is a generalization of Makowsky and Rakita's question. For the case of $l=2$, we conjecture:
\begin{con}\label{con:LCof2matroids}
Let $\mathcal{M}$ be an $n$-element $2$-matroid, $i_k$ the number of independent sets of size $k$ in $\mathcal{M}$. Then
\begin{enumerate}
\renewcommand{\labelenumi}{(\roman{enumi})}
\item the sequence $\{i_k\}_{k=0}^{n}$ is log-concave;
\item the sequence $\{i_k\}_{k=0}^{n}$ is ordered log-concave;
\item the sequence $\{i_k\}_{k=0}^{n}$ is ultra log-concave.
\end{enumerate}
\end{con}

Recall $\overline{\mathcal{D}}$ the family of co-dependent sets in a graph $G$. Then $(V(G),\overline{\mathcal{D}})$ is a 2-matroid. Thus, Conjecture \ref{con:LCof2matroids} is a stronger version of Conjecture \ref{con:ULCofDependencePolynomials}.

It is remarkable that for $3$-matroids, the analogous conjecture is not true. Let's give a counterexample:
\begin{exa}\label{exa:CounterExample}
Assume $\mathcal{M}$ is a $3$-matroid on a $10$-element ground set, where $\mathcal{I}:=\{\emptyset,\{1\},\{2\},$ $\{3\},\{4\},\{5\},\{6\},\{7\},\{8\},\{9\},\{10\},\{1,2\},\{1,3\},\{2,3\},\{1,2,3\}\}$. Then $i_1=10$, $i_2=3$ and $i_3=1$. $i_2^2<i_1i_3$, so the sequence $\{i_k\}_{k=0}^{10}$ is not log-concave.
\end{exa}

\vskip 0.2 cm
\noindent{\bf Acknowledgements:} This work is partially supported by National Natural Science Foundation of China (Grants No. 12671410).

\end{document}